\documentclass[11pt]{article}
\usepackage[english]{babel}
\usepackage{epsfig}
\usepackage{amsmath}
\usepackage{amsthm}
\usepackage{amssymb}
\usepackage{mathrsfs}
\usepackage[mathscr]{euscript}
\usepackage{color}

\newcommand{\ux}{\underline{x}}

\newtheorem{theorem}{Theorem}[section]
\newtheorem{definition}[theorem]{Definition}

\newtheorem{lemma}[theorem]{Lemma}
\newtheorem{proposition}[theorem]{Proposition}

\newtheorem{remark}[theorem]{Remark}

\newcommand{\rr}{{\mathbb{R}}}

\newcommand{\unx}{\underline{x}}
\newcommand{\unz}{\underline{z}}

\begin{document}

\title{Segal-Bargmann-Fock modules of monogenic functions\thanks{Published in Journal of Mathematical Physics 58 (2017), no. 10, https://doi.org/10.1063/1.5008651.}}

\author{Dixan Pe\~na Pe\~na\\
\small{Dipartimento di Matematica}\\
\small{Politecnico di Milano}\\
\small{Via E. Bonardi 9}\\
\small{20133 Milano, Italy}\\
\small{e-mail: dixanpena@gmail.com}\and
Irene Sabadini\\
\small{Dipartimento di Matematica}\\
\small{Politecnico di Milano}\\
\small{Via E. Bonardi 9}\\
\small{20133 Milano, Italy}\\
\small{e-mail: irene.sabadini@polimi.it}\and
Franciscus Sommen\\
\small{Clifford Research Group}\\
\small{Department of Mathematical Analysis}\\
\small{Faculty of Engineering and Architecture}\\
\small{Ghent University}\\
\small{Krijgslaan 281-S8, 9000 Gent, Belgium}\\
\small{e-mail: franciscus.sommen@ugent.be}}

\date{}

\maketitle

\begin{abstract}
\noindent In this paper we introduce the classical Segal-Bargmann transform starting from the basis of Hermite polynomials and extend it to Clifford algebra-valued functions. Then we apply the results to monogenic functions and prove that the Segal-Bargmann kernel corresponds to the kernel of the Fourier-Borel transform for monogenic functionals. This kernel is also the reproducing kernel for the monogenic Bargmann module.
\vspace{0.2cm}\\
\noindent\textit{Keywords}: Segal-Bargmann-Fock spaces; Segal-Bargmann transform; monogenic functions; Dirac operators.\vspace{0.1cm}\\
\textit{Mathematics Subject Classification}: 30G35, 42B10, 44A15.
\end{abstract}

\section{Introduction}

The Segal-Bargmann transform can be seen as a unitary map from spaces of square-integrable functions to spaces of  holomorphic functions square-integrable with respect to a Gaussian density (see \cite{Ba}, \cite{Se1}, \cite{Se2}). The latter class of functions is known in the literature under different names: Bargmann or Segal-Bargmann or Fock spaces either in one or several complex variables (see \cite{folland}, \cite{neretin}, \cite{zhu}). In this paper we do not enter historical discussion on the nomenclature and we will call them Segal-Bargmann-Fock spaces. These spaces are important in several different settings:
in quantum mechanics, where they it is used the description of the spaces via tensor products; in infinite dimensional analysis and in free analysis, since these spaces are related to the white noise space (a probability space) and to the theory of stochastic distributions, see \cite{dan}.

It is natural to consider another higher dimensional extension, namely the one based on monogenic functions with values in a  Clifford algebra. Recently, in \cite{kmnq}, \cite{mnq}, the authors have pointed out that this approach can be useful in studying quantum systems with internal, discrete degrees of freedom corresponding to nonzero spins.

Square-integrable holomorphic functions are well-known in the literature and they may belong to Hardy or Bergman spaces, whose reproducing kernels are rational functions. In the case of the Segal-Bargmann-Fock spaces the kernel is an exponential function. This reproducing kernel is not bounded and this fact makes some computations in this framework more complicated. On the other hand, the fact that Segal-Bargmann-Fock spaces are defined on the whole space (either $\mathbb C$ or $\mathbb C^m$) makes other techniques from Fourier analysis available. From the physics point of view, the (normalized) reproducing kernels of the Segal-Bargmann-Fock spaces are the so-called coherent states in quantum mechanics.

In this paper we consider the higher dimensional framework based on monogenic functions. We can introduce a notion of Segal-Bargmann module (over the Clifford algebra) and of Segal-Bargmann-Fock transform. The fact that we have a Fischer decomposition allows to prove a relation between the projection of the transform onto its monogenic part and the Fourier-Borel kernel. Equipping the Segal-Bargmann-Fock module with the Fischer inner product we can also show that the Segal-Bargmann transform is still an isometry on harmonic polynomials.

The plan of the paper is the following. In Section 2 we revise the classical Segal-Bargmann transform. Then in Section 3 we apply the results of Section 2 to the case of monogenic functions. In Section 4 we study monogenic Bargmann modules. Finally in Section 5 we make summarizing remarks.

\section{The classical Segal-Bargmann transform}

The Segal-Bargmann transform is a very well-known operator acting from the space $L^2(\rr)$ and the Segal-Bargmann-Fock space which is unitary and allows to identify these two spaces (see \cite{Ba}, \cite{Se1}, \cite{Se2} and the recent monographs \cite{neretin}, \cite{zhu}).\\
The Segal-Bargmann-Fock space $\mathscr B(\mathbb C)$ is the Hilbert space of entire functions which are  square-integrable with respect to the Gaussian density, i.e.
\[
\frac{1}{\pi}\int_{\mathbb C}\exp\left(-|z|^2\right) |f(z)|^2 \, dx dy<\infty.
\]
For any $f\in L^2(\rr)$, we define the Segal-Bargmann transform $\mathcal B:\ L^2(\mathbb R)\to \mathscr B(\mathbb C)$ (sometimes called Segal-Bergmann-Fock transform) as
\[
\mathcal B[f](z)= \frac{1}{\sqrt{2\pi}}\int_{\rr} \exp\left(-\frac{z^2}{2}+xz-\frac{x^2}{4}\right) f(x) \, dx,
\]
for any $f\in L^2(\mathbb R)$.\\
Let us consider in $L^2(\rr)$ the basis $\{\psi_k(x)\}$, where
\[
\psi_k(x)=H_k(x)e^{-{x^2}/{4}}, \qquad H_k(x)=(-1)^ke^{{x^2}/{2}}\partial_x^k e^{-{x^2}/{2}}.
\]
Note that $H_k(x)$ are  the well-known Hermite polynomials. This basis is orthogonal, in fact
\[
\langle\psi_\ell,\psi_k\rangle=\frac{1}{\sqrt{2\pi}}\int_{\rr} \overline{\psi_\ell(x)}\psi_k(x)\, dx=\delta_{\ell k}k!,
\]
where $\delta_{\ell k}$ is the Kronecker delta.

Using integration by parts we obtain
\[
\begin{split}
\mathcal B[\psi_k](z)&=\frac{1}{\sqrt{2\pi}}\int_{\rr} \exp\left(-\frac{(z-x)^2}{2}\right) H_k(x)\, dx\\
&=\frac{1}{\sqrt{2\pi}}\int_{\rr}\partial_x^k \left(\exp\left(-\frac{z^2}{2}+xz\right)\right) \exp(-x^2/2)\, dx\\
&=\frac{1}{\sqrt{2\pi}}z^k \int_{\rr}\exp\left(-\frac{(z-x)^2}{2}\right)\,dx\\
&= z^k.
\end{split}
\]
Let us consider the so-called Fischer space, namely the space of real analytic functions equipped with the inner product
\[
[ R,S]= \overline{R(\partial_x)}S(x)\vert_{x=0}.
\]
To consider functions defined over the real numbers is necessary to obtain a match between the Fischer and the Bargmann inner products.\\
If we consider instead of the complex variable $z$ a real variable $u$, we can use the Fischer inner product to get the equality $\left[ u^\ell, u^k \right] = \delta_{\ell k}k!$.

The above computations show that the Segal-Bargmann-Fock transform $f(x)\mapsto \mathcal B[f](u)$ is an isometry between $L^2(\rr)$ and the Fischer space.

On the other hand we can consider the space $\mathscr O(\mathbb C)$ of entire holomorphic functions equipped with the inner product
\[
\langle f, g\rangle=\frac{1}{\pi}\int_{\mathbb C}\exp\left(-|z|^2\right)\overline{f(z)}g(z) \, dx dy, \qquad z=x+iy.
\]
We note that the basis $\left\{z^k\right\}$ is orthogonal, in fact:
\[
\langle z^\ell, z^k\rangle = \frac{1}{\pi}\int_{\mathbb C}\exp\left(-|z|^2\right)\overline{z}^k z^\ell \, dx dy=0\quad{\rm for}\ k\not=\ell.
\]
Moreover
\[
\langle z^k, z^k\rangle = \frac{1}{\pi}\int_0^{2\pi}\int_{0}^\infty\exp\left(-r^2\right)r^{2k}\,rdrd\theta = k!,\quad{\rm where}\ z=re^{\theta}.
\]
Then we have that $\mathcal B: \ L^2(\rr)\to \mathscr B(\mathbb C)$ is an isometry. Indeed, the Fischer inner product for functions defined in $\mathbb R$ is equal to the Bargmann inner product of the holomorphic extensions.

This discussion can be extended from one to several complex variables.
\begin{definition}
The Segal-Bargmann-Fock space $\mathscr B(\mathbb C^m)$ is defined as the Hilbert space of entire functions $f$ in $\mathbb C^m$ which are square-integrable with respect to the $2m$-dimensional Gaussian density, i.e.
\[
\frac{1}{\pi^m}\int_{\mathbb C^m}\exp\left(-|\unz|^2\right) |f(\unz)|^2 \, d\unx d\underline{y}<\infty, \qquad \unz=\unx+i\underline{y}
\]
and equipped with the inner product
\[
\langle f,g\rangle =\frac{1}{\pi^m}\int_{\mathbb C^m} \exp\left(-|\unz|^2\right)\overline{f(\unz)}g(\unz)\, d\unx d\underline{y}.
\]
\end{definition}

Let us denote by $\unx$ the $m$-tuple $(x_1,\ldots, x_m)\in\rr^m$. We consider the space $L^2(\rr^m)$ with the basis $\{\psi_{k_1\ldots k_m}(\unx)\}$ where
\[
\psi_{k_1\ldots k_m}(\unx)= H_{k_1\ldots k_m}(\unx) e^{-|\unx|^2/4}
\]
and $H_{k_1\ldots k_m}$ are the Hermite polynomials in $\mathbb R^m$ given by
\[H_{k_1\ldots k_m}(\unx) e^{-\vert\unx\vert^2/2}=(-1)^{k_1+\dots+k_m}\partial_{x_1}^{k_1}\ldots \partial_{x_m}^{k_m} e^{-\vert\unx\vert^2/2}.\]

\begin{definition}\label{SBtr}
The Segal-Bargmann-Fock transform $\mathcal B : \  L^2(\rr^m) \to \mathcal B(\mathbb C^m)$ is defined by
\[
\mathcal B[f](\unz)=\frac{1}{({2\pi})^{m/2}}\int_{\mathbb R^m}\exp\left(-\frac{\unz\cdot\unz}{2}+\unx\cdot\unz-\frac{\unx\cdot\unx}{4}\right) f(\underline x) \, d\unx,\quad \unx\cdot\unz=\sum_{j=1}^mx_jz_j,
\]
for any  $f\in L^2(\rr^m)$.
\end{definition}
\begin{remark}{\rm
Similarly to the one variable case we have that
\[
\mathcal B[\psi_{k_1\ldots k_m}](\unz)=z_1^{k_1}\ldots z_m^{k_m}.
\]
}
\end{remark}
Finally note that $\|\psi_{k_1\ldots k_m}\|^2=k_1!\ldots k_m!$. It follows from this that $\mathcal B: \ L^2(\rr^m)\to \mathscr B(\mathbb C^m)$ is an isometry.

\section{The monogenic case}

Let us denote by $\rr_m$ the real Clifford algebra generated by $m$ imaginary units $e_1,\dots,e_m$. The multiplication in this associative algebra is determined by the relations
\[e_je_k+e_ke_j=-2\delta_{ij}.\]
An element $a\in\rr_m$ can be written as
\[a=\sum_Aa_Ae_A,\quad x_A\in\rr,\]
where the basis elements $e_A=e_{j_1}\dots e_{j_k}$ are defined for every subset $A=\{j_1,\dots,j_k\}$ of $\{1,\dots,m\}$ with $j_1<\dots<j_k$. For the empty set, one puts $e_{\emptyset}=1$, the latter being the identity element.

Observe that the dimension of $\mathbb R_{m}$ as a real linear space is $2^m$. Furthermore, conjugation in $\mathbb R_{m}$ is given by $\overline a=\sum_Aa_A\overline e_A$, where $\overline e_A=\overline e_{j_k}\dots\overline e_{j_1}$ with $\overline e_j=-e_j$, $j=1,\dots,m$.

In the Clifford algebra $\mathbb{R}_m$, we can identify the so-called 1-vectors, namely the linear combinations with real coefficients of the elements $e_j$, $j=1,\ldots,m$, with the vectors in the Euclidean space $\mathbb{R}^m$. The correspondence is given by the map $(x_1,\ldots,x_m)\mapsto \underline{x}=x_1e_1+\ldots+x_me_m$ and it is obviously one-to-one.

The norm of a 1-vector $\underline x$ is defined as $\vert\unx\vert=\sqrt{x_1^2+\ldots +x_m^2}$ and clearly $\underline x^2=-\vert\unx\vert^2$. The product of two 1-vectors $\underline x=\sum_{j=1}^mx_je_j$ and $\underline y=\sum_{j=1}^my_je_j$ splits into a scalar part and a 2-vector or so-called bivector part:
$$
\ux\, \underline{y}=-\langle\ux ,\underline{y}\rangle+\ux\wedge\underline{y},
$$
where
\[
\begin{split}
&\langle\ux ,\underline{y}\rangle=\unx\cdot\underline{y}=-\frac{1}{2}(\ux\,\underline{y}+\underline{y}\,\ux)=\sum_{j=1}^m x_jy_j\\
 &\ux\wedge\underline{y}=\frac{1}{2}(\ux\,\underline{y}-\underline{y}\,\ux)=\sum_{j=1}^m\sum_{k=j+1}^m(x_jy_k-x_ky_j)e_je_k.
 \end{split}
\]
 The complex Clifford algebra $\mathbb C_m$ can be seen as the complexification of the real Clifford algebra $\mathbb R_{m}$, i.e. $\mathbb C_m=\mathbb R_{m}\oplus i\,\mathbb R_{m}$. Any complex Clifford number $c\in\mathbb C_m$ may be written as $c=a+ib$, $a,b\in\mathbb R_{m}$, leading to the definition of the Hermitian conjugation: $c^{\dagger}=\overline a-i\overline b$.

The first order differential operator
\[\partial_{\underline x}=\sum_{j=1}^me_j\partial_{x_j}\]
is called the Dirac operator in $\mathbb R^m$. Functions in the kernel of this operator are known as monogenic functions (see e.g. \cite{bds}, \cite{csss}, \cite{DSS}, \cite{gm}).

\begin{definition}
A function $f:\Omega\subset\mathbb R^{m}\rightarrow\mathbb{C}_{m}$ defined and continuously differentiable in the open set $\Omega$ is said to be $($left$)$ monogenic if $\partial_{\underline x}f(\underline x)=0$ in $\Omega$. We denote by $\mathscr M(\rr^m)$ the right $\mathbb C_m$-module of monogenic functions in $\mathbb R^m$.
\end{definition}

A basic result in Clifford analysis is the so-called Fischer decomposition. Every homogeneous polynomial $R_k$ of degree $k$ can be uniquely decomposed as
\[R_k(\underline x)=M_k(\underline x)+\underline xR_{k-1}(\underline x),\]
where $M_k$, $R_{k-1}$ are homogeneous polynomials and $M_k\in\mathscr M(\rr^m)$. The monogenic polynomial $M_k$ is called the monogenic part of $R_k$ denoted by $\mathcal M(R_k)$.

If $f$ is real analytic function near the origin, then it admits a decomposition of the form $f(\underline x)=\sum_{k=0}^\infty R_k(\underline x)$ in an open ball centred at the origin. The monogenic part of $f$ is thus defined by
\[\mathcal M\big(f(\underline x)\big)=\sum_{k=0}^\infty \mathcal M\big(R_k(\underline x)\big).\]
Let us recall the definition of the Clifford-Hermite polynomials $H_{s,k}(\underline x)$. They are polynomials with real coefficients in $\underline x$ of degree $s$ satisfying
\[H_{s,k}(\unx)e^{-|\unx|^2/2}P_k(\unx)=(-1)^s\partial_{\unx}^s\big(P_k(\unx)e^{-|\unx|^2/2}\big),\]
where $P_k(\unx)$ denotes a homogeneous polynomial of degree $k$ in $\mathscr M(\rr^m)$.

Put $\psi_{s,k}(\unx)=H_{s,k}(\unx)e^{-|\unx|^2/4}$.  The set of functions $\big\{\psi_{s,k}(\unx)P_k(\unx):\,s,k\in\mathbb N\big\}$ is an orthogonal basis for $L^2(\rr^m)$ (see \cite{Som}). Recalling Definition \ref{SBtr}, we have the following:

\begin{theorem}\label{Th3.2}
Assume that $P_k(\unx)$ is a homogeneous polynomial of degree $k$ in $\mathscr M(\rr^m)$. Then the following formula holds:
\[\mathcal B\big[\psi_{s,k}(\underline x)P_k(\underline x)\big](\unz) = \unz^s \mathcal B\big[P_k(\unx)e^{-|\unx|^2/4}\big](\underline z).\]
\end{theorem}
\begin{proof} The formula follows using integration by parts in higher dimensions:
\[
\begin{split}
\mathcal B\big[\psi_{s,k}(\underline x)P_k(\underline x)\big](\unz)&=\frac{1}{(2\pi)^{m/2}}\int_{\rr^m} \exp\left(-\frac{\unz\cdot\unz}{2}+\unx\cdot\unz-\frac{\unx\cdot\unx}{4}\right) \psi_{s,k}(\unx)P_k(\unx)\, d\unx\\
&=\frac{(-1)^s}{(2\pi)^{m/2}}\int_{\rr^m} \exp\left(-\frac{\unz\cdot\unz}{2}+\unx\cdot\unz\right) \partial_{\unx}^s\big(P_k(\unx) e^{-|\unx|^2/2}\big)\, d\unx \\
&=\frac{1}{(2\pi)^{m/2}}\int_{\rr^m} \partial_{\unx}^s\left(\exp\left(-\frac{\unz\cdot\unz}{2}+\unx\cdot\unz\right)\right)P_k(\unx) e^{-|\unx|^2/2}\, d\unx \\
&= \unz^s\mathcal B\big[P_k(\unx)e^{-|\unx|^2/4}\big](\underline z).
\end{split}
\]
\end{proof}
\noindent The next problem is to compute $\mathcal B\big[P_k(\unx)e^{-|\unx|^2/4}\big]$. To this end we consider an example from which the general result will follow.

\begin{lemma}\label{ex1}
Let $P_k(\unx)=(x_1-e_1e_2 x_2)^k$. Then
\[
\mathcal B\big[P_k(\unx)e^{-|\unx|^2/4}\big](\unz)=(z_1 -e_1e_2z_2)^k=P_k(\underline{z}).
\]
\end{lemma}
\begin{proof}
Let  $Z=x_1-e_1e_2 x_2$ and $2\partial_{\bar Z}=\partial_{x_1}-e_1e_2\partial_{x_2}$. Note that
\[
e^{-|\unx|^2/2}=\exp\left(-\frac 12 Z \bar Z\right)\exp\left( \sum_{j=3}^m - \frac 12 x_j^2\right).
\]
Thus we have
\[
\begin{split}
Z^k e^{-|x|^2/2}&=\left(\left(-2\partial_{\bar Z}\right)^k  \exp\left(-\frac 12 Z \bar Z\right)\right)\exp\left( \sum_{j=3}^m - \frac 12 x_j^2\right)\\
&=\big(-(\partial_{x_1}-e_1e_2 \partial_{x_2})\big)^k e^{-|\unx|^2/2}.
\end{split}
\]
Therefore
\[
\begin{split}
\mathcal B\big[P_k(\unx)e^{-|\unx|^2/4}\big](\unz)&=\frac{1}{(2\pi)^{m/2}}\int_{\rr^m}\left((\partial_{x_1} -e_1e_2\partial_{x_2})^k \exp\left(-\frac{\unz\cdot\unz}{2} +\unx\cdot \unz\right)\right)e^{-|\unx|^2/2}d\unx\\
&=\frac{(z_1 -e_1e_2z_2)^k}{(2\pi)^{m/2}}\int_{\rr^m}\exp\left(-\frac{(\unz-\unx)\cdot(\unz-\unx)}{2}\right)d\unx\\
&=\frac{(z_1 -e_1e_2z_2)^k}{(2\pi)^{m/2}}\int_{\rr^m}\exp\left(-\frac{ |\unx|^2}{2}\right)d\unx\\
&=(z_1 -e_1e_2z_2)^k.
\end{split}
\]
\end{proof}
\noindent We can now prove the following simple but important result:

\begin{theorem}\label{fundteoFra}
Suppose that $P_k(\unx)$ is a homogeneous polynomial of degree $k$ in $\mathscr M(\rr^m)$. Then the following formula holds:
\[
\mathcal B\big[H_{s,k}(\unx)e^{-|\unx|^2/4}P_k(\unx)\big](\unz)=\unz^s P_k(\unz) .
\]
\end{theorem}
\begin{proof}
We show that the formula that we have established for $P_k(\unx)=(x_1 -e_1e_2 x_2)^k$ in Lemma \ref{ex1} holds for a general $P_k(\unx)$. First of all, due to the Spin$(m)$-invariance, the result also holds for monogenic plane waves functions of the form
\[
(\langle \unx, \underline t\rangle -\underline t\, \underline s \langle \unx, \underline s\rangle)^k,
\]
where $\underline t$ and $\underline s$ are orthogonal unit vectors. Then we notice that the space of homogeneous monogenic polynomials of degree $k$ is spanned by finitely many monogenic plane waves $(\langle \unx, \underline t\rangle -\underline t\, \underline s \langle \unx, \underline s\rangle)^k$ for some choices of the parameters $(\underline t,\underline s)$. This, combined with Theorem \ref{Th3.2}, proves the result.
\end{proof}

\section{Monogenic Bargmann modules}

We will use the following notations.

\begin{definition}
Let $s\in\mathbb N$. We denote by $\mathscr M^s(\rr^m)$ the right $\mathbb C_m$-module of $s$-monogenic functions on $\rr^m$, namely the set of smooth functions in the kernel of $\partial_{\unx}^s$. Next, by $\mathscr M^s(\mathbb C^m)$ we denote the right $\mathbb C_m$-module of $s$-monogenic functions on $\mathbb C^m$, namely the set of entire holomorphic functions in the kernel of $\partial_{\unz}^s$, where $\partial_{\underline z}=\sum_{j=1}^me_j\partial_{z_j}$ is the complexified Dirac operator.
\end{definition}

Note that $\mathscr M^s(\mathbb C^m)$ is spanned by the set of polynomials of the form $\unz^jP_k(\unz)$, $j=0,\ldots ,s-1$, where $P_k(\unz)$ is complex spherical monogenics of degree $k\in\mathbb N$.

Let us now define the $s$-monogenic Bargmann modules.

\begin{definition}
For any $s\in\mathbb N$,
the s-monogenic Bargmann module $\mathscr{MB}^s(\mathbb C^m)$ is defined as
\[
\mathscr{MB}^s(\mathbb C^m)=\mathscr M^s(\mathbb C^m)\cap \mathscr B(\mathbb C^m),
\]
and it is is equipped with the inner product defined in $\mathscr B(\mathbb C^m)$.
\end{definition}
\noindent The Segal-Bargmann transform may act on $s$-monogenic functions as described in the following:

\begin{proposition}
The map
\[
\mathcal B: \ \mathscr M^s(\mathbb R^m)e^{-|\unx|^2/4} \cap L^2(\rr^m)\to \mathscr{MB}^s(\mathbb C^m)
\]
is an isometry (and it is even unitary).
\end{proposition}
\begin{proof}
The proof follows the standard arguments in the complex case and since $H_{s,k}(\unx)P_k(\unx)$, $j=0,\ldots ,s-1$, span $\mathscr M^s(\mathbb R^m)$, it is a consequence of the previous Theorem \ref{fundteoFra}.
\end{proof}
\begin{remark}{\rm
From the Fischer decomposition we obtain
\begin{equation*}
\frac{\langle\underline x,\underline u\rangle^k}{k!}=\sum_{s=0}^k\underline x^sZ_{k,s}(\underline x,\underline u)\,\underline u^s,
\end{equation*}
where $Z_{k,s}(\underline x,\underline u)$ are the so-called zonal spherical monogenics (see \cite{SoJa}). The latter are homogeneous polynomials of degree $k-s$ in $\underline x$ and $\underline u$ and satisfy the two-sided biregular system $\partial_{\underline x}Z_{k,s}(\underline x,\underline u)=Z_{k,s}(\underline x,\underline u)\partial_{\underline u}=0$.
\\
One can also check that
\[Z_{k,s}(\underline x,\underline u)=\frac{Z_{k-s,0}(\ux,\underline u)}{\beta_{s,k-s}\ldots \beta_{1,k-s}},\quad k\ge s\]
with $\beta_{2s,k}=-2s$, $\beta_{2s+1,k}=-(2s+2k+m)$. Moreover,
\begin{multline*}
Z_{k}(\underline x,\underline u)=Z_{k,0}(\underline x,\underline u)\\
=\frac{\Gamma\left(\frac{m}{2}-1\right)}{2^{k+1}\Gamma\left(k+\frac{m}{2}\right)}(\vert\underline x\vert\vert\underline u\vert)^k\left[(k+m-2)C_k^{\frac{m}{2}-1}(t)+(m-2)\frac{\underline x\wedge\underline u}{\vert\underline x\vert\vert\underline u\vert}C_{k-1}^{\frac{m}{2}}(t)\right],
\end{multline*}
where $C_k^\alpha(t)$ denotes the classical Gegenbauer polynomial and $t=\displaystyle{\frac{\langle\underline x,\underline u\rangle}{\vert\underline x\vert\vert\underline u\vert}}$.
}
\end{remark}
\begin{remark}{\rm
Therefore, the Fischer decomposition of $e^{\langle\underline x,\underline u\rangle}$ has the form
\begin{equation*}
e^{\langle\underline x,\underline u\rangle}=\sum_{k=0}^\infty\frac{\langle\underline x,\underline u\rangle^k}{k!}=\sum_{k=0}^\infty\sum_{s=0}^k\underline x^sZ_{k,s}(\underline x,\underline u)\,\underline u^s=\sum_{s=0}^\infty\underline x^sE_s(\underline x,\underline u)\,\underline u^s,
\end{equation*}
with $E_s(\underline x,\underline u)=\sum_{k=s}^\infty Z_{k,s}(\underline x,\underline u)$.
}
\end{remark}

We note that $E(\underline x,\underline u)=E_0(\underline x,\underline u)$ is the monogenic part of $e^{\langle\underline x,\underline u\rangle}$, which is known as the Fourier-Borel kernel. This function has been computed in closed form using Bessel functions in \cite{ndss} (see also \cite{ss}).

Let us again denote by $\mathcal M$ the projection of a function onto its monogenic part in the complex monogenic Fischer decomposition.

\begin{proposition}
Let $f(\unx)=g(\unx)e^{-|\unx|^2/4}\in L^2(\rr^m)$. Then we have
\[
\mathcal M\big(\mathcal B[f](\unz)\big)=\frac{1}{(2\pi)^{m/2}}\int_{\rr^m}E(\unz,\unx) g(\unx) e^{-|\unx|^2/2}\, d\unx=[E(\unz,\unx)^\dagger, g(\unx)],
\]
where $[\cdot, \cdot]$ denotes the Fischer inner product. In particular when $g$ is monogenic, we obtain $\mathcal M\big(\mathcal B[f](\unz)\big)=g(\unz)$.
\end{proposition}
\begin{proof}
From the above remarks we obtain
\[
\mathcal M \left( \exp\left(-\frac{\unz\cdot\unz}{2}+\unx\cdot\unz-\frac{\unx\cdot\unx}{4}\right)\right)=\mathcal M\big(\exp(\unx\cdot\unz)\big)e^{-|\unx|^2/4}=E(\unz,\unx)e^{-|\unx|^2/4}.
\]
The statement follows using
\[
\mathcal M\big(\mathcal B[f](\unz)\big)=\frac{1}{(2\pi)^{m/2}} \int_{\rr^m} \mathcal M\left(\exp\left(-\frac{\unz\cdot\unz}{2}+\unx\cdot\unz-\frac{\unx\cdot\unx}{4}\right)\right)f(\unx)\, d\unx.
\]
\end{proof}

\section{Real $s$-monogenic Bargmann modules}

Let us recall that for harmonic polynomials $R(\unx),S(\unx)$ we have (see also \cite{gm})
\[
[R(\unx), S(\unx)]=R(\partial_{\unx})^\dagger S(\unx)\vert_{\unx=\underline 0}=\frac{1}{(2\pi)^{m/2}}\int_{\rr^m}
R(\unx)^\dagger S(\unx)\,e^{-|\unx|^2/2}\,d\unx.
\]
For $P_k(\unx)\in\mathscr M(\rr^m)$ we have
\[
P_k(\underline u)=[Z_k(\underline u,\unx)^\dagger, P_k(\underline x)]=[E(\underline u,\unx)^\dagger, P_k(\unx)].
\]
From this formula one also obtains in another way that for every monogenic $g(\unx)$ such that $g(\unx)e^{-|\unx|^2/4}\in L^2(\rr^m)$ the following formula holds
\[
g(\underline u)=\frac{1}{(2\pi)^{m/2}}\int_{\rr^m} E(\underline u,\unx)g(\unx) e^{-|\unx|^2/2}\, d\unx.
\]
For general polynomials $R(\unx),S(\unx)$ this link with the Fischer inner product no longer holds. However, we have that
\[
[R(\unx), S(\unx)]=\frac{1}{\pi^m}\int_{\mathbb C^m}R(\unz)^\dagger S(\unz)e^{-|\unz|^2}\, d\unz .
\]
For instance, take $R,S$ of the form $z_1^{k_1}\ldots z_m^{k_m}$. Moreover, $\mathcal B$ is an isometry between $L^2(\rr^m)$ and the Segal-Bargmann-Fock module equipped with the above inner product. In particular, if $g\in\mathscr M^s(\rr^m)$ we have that the map
\[
\mathcal B:\ g(\unx)e^{-|\unx|^2/4} \mapsto h(\unz)=\mathcal B ( g(\unx)e^{-|\unx|^2/4})
\]
is an isometry. Note also that $h$ is $s$-monogenic but it is no longer true that $h(\underline u)=g(\underline u)$ for $\underline u\in\rr^m$.

Let us consider $f\in \mathscr B(\mathbb C^m)$ and let $\mathcal M_s$ be the orthogonal projection onto the submodule of $s$-monogenic functions $\mathscr{MB}^s(\mathbb C^m)$. To be more precise, if
\[
f(\unz)=\sum_{j=0}^\infty \unz^j f_j(\unz),\quad\partial_{\unz}f_j(\unz)=0,
\]
is the Fischer decomposition of $f$, then $\mathcal M_s\big(f(\unz)\big)=\sum_{j=0}^{s-1} \unz^jf_j(\unz)$.

Define the function
\[
B_s(\unz,\unx ) = \mathcal M_s\left(\exp\left( -\frac{\unz\cdot \unz}{2} +\unx\cdot\unz\right) \right).
\]

\begin{theorem}
The following formula holds:
\[
B_s(\unz,\unx )=\sum_{j=0}^{s-1}\unz^j \sum_{\ell=0}^{[j/2]} \frac{1}{2^\ell\ell!}  E_{j-2\ell}(\unz ,\unx) \unx^{j-2\ell}.
\]
\end{theorem}
\begin{proof}
We have that
\[
\exp\left( -\frac{\unz\cdot \unz}{2}\right)=\sum_{\ell=0}^\infty \frac{\unz^{2\ell}}{2^\ell\ell!}
\]
and
\[
\exp\left(\unx\cdot\unz\right)=\sum_{k=0}^\infty \unz^k E_k(\unz ,\unx) \unx^k.
\]
Therefore
\[
\begin{split}
\exp\left( -\frac{\unz\cdot \unz}{2} +\unx\cdot\unz\right)&=\sum_{k,\ell=0}^\infty \frac{1}{2^\ell\ell!} \unz^{k+2\ell}E_k(\unz ,\unx) \unx^k \\
&= \sum_{j=0}^\infty \unz^j \sum_{\ell=0}^{[j/2]} \frac{1}{2^\ell\ell!}  E_{j-2\ell}(\unz ,\unx) \unx^{j-2\ell}
\end{split}
\]
from which the desired result easily follows.
\end{proof}

\begin{remark}
{\rm Note also the formula
\[
\mathcal M_s\left(\mathcal B\big[g(\unx) e^{-|\unx|^2/4}\big](\underline z)\right)=\frac{1}{(2\pi)^{m/2}}\int_{\rr^m} \mathcal M_s \left(\exp\left( -\frac{\unz\cdot \unz}{2} +\unx\cdot\unz\right)\right)g(\unx) e^{-|\unx|^2/2}\, d\unx.
\]
}
\end{remark}
\noindent Next, one may introduce the real $s$-monogenic Bargmann module, which is the set of $g\in\mathscr M^s(\rr^m)$ such that $g(\unx)e^{-|\unx|^2/4}$ is square-integrable.
\begin{definition}
For any $s\in\mathbb N$, the real s-monogenic Bargmann module $\mathscr{MB}^s(\rr^m)$ is defined by
\[
\mathscr{MB}^s(\rr^m)=\mathscr M^s(\rr^m)\cap L^2\big(\rr^m, e^{-|\unx|^2/2}\big).
\]
It is equipped with the inner product in  $L^2\big(\rr^m, e^{-|\unx|^2/2}\big)$
\[
\langle f,g\rangle = \frac{1}{(2\pi)^{m/2}}\int_{\rr^m}f(\unx)^\dagger g(\unx)e^{-|\unx|^2/2} \, d\unx
\]
\end{definition}
As an additional remark, we note that for any $f\in \mathscr{MB}^s(\rr^m)$ we can define its Weierstrass transform $\mathcal W:\ \mathscr{MB}^s(\rr^m)\to \mathscr{MB}^s(\mathbb C^m)$ by
\[
\mathcal W [f](\unz)=\frac{1}{(2\pi)^{m/2}} \int_{\rr^m}\exp\left(  -\frac{(\unz-\unx)\cdot (\unz-\unx)}{2}\right) f(\unx)\, d\unx=\mathcal B\big[f(\unx) e^{-|\unx|^2/4}\big](\underline z).
\]
With standard techniques, one can show that the map $\mathcal W:\ \mathscr{MB}^s(\rr^m)\to \mathscr{MB}^s(\mathbb C^m)$ is an isometry, moreover
\[
\mathcal W[f](\unz) = \frac{1}{(2\pi)^{m/2}} \int_{\rr^m} B_s(\unz,\unx) f(\unx) e^{-|\unx|^2/2}\, d\unx.
\]

\subsection*{Acknowledgments}

D. Pe\~na Pe\~na acknowledges the support of a Postdoctoral Fellowship given by Istituto Nazionale di Alta Matematica (INdAM) and cofunded by Marie Curie actions.

\end{document}